\documentclass{amsart}

\usepackage{amsmath,amssymb,amsthm,graphicx, paralist, appendix}
\usepackage{hyperref}

\usepackage{subfigure}
\usepackage{amsrefs}
\usepackage[arrow, matrix, curve]{xy}
\usepackage{stmaryrd} 
\usepackage{color} 

\def\co{\colon\thinspace}

\newcommand{\id}{\mbox{\rm id}}
\newcommand{\e}{\mathrm{e}}

\newcommand{\C}{\mathbb{C}}

\newcommand{\R}{\mathbb{R}}

\newtheorem{thm}{Theorem}
\newtheorem*{thm*}{Theorem}

\newtheorem{prop}[thm]{Proposition}
\newtheorem{cor}[thm]{Corollary}

\newtheorem*{claim*}{Claim}

\theoremstyle{definition}

\newtheorem*{rem*}{Remark}
\newtheorem*{rems*}{Remarks}
\newtheorem*{remdefn*}{Remark/Definition}

\newtheorem*{defn*}{Definition}
\newtheorem*{ex*}{Example}

\newtheorem*{que*}{Question}

\begin{document}
\title[]{Open books and exact symplectic cobordisms}
\author{Mirko Klukas}
\address{Mathematisches Institut, University of Cologne, Weyertal 86--90,
D-50931 Cologne, Germany}
\email{mklukas@math.uni-koeln.de}
\date{\today}
\maketitle
\begin{abstract}
Given two open books with equal pages we show the existence of an exact symplectic cobordism 
whose negative end equals the disjoint union of the contact manifolds associated to the given open books, 
and whose positive end induces the contact manifold associated to the open book with the same 
page and concatenated monodromy.
Using similar methods we show the existence of strong fillings for contact manifolds 
associated with \textit{doubled open books}, a certain class of fiber bundles over the circle obtained 
by performing the binding sum of two open books with equal pages and inverse 
monodromies. From this we conclude, following an outline by Wendl, that the 
complement of the binding of an open book cannot contain any local filling obstruction.
Given a contact $3$-manifold, according to Eliashberg there is a symplectic cobordism to a 
fibration over the circle with symplectic fibers. We extend this result to higher 
dimensions recovering a recent result by D\"orner--Geiges--Zehmisch. Our cobordisms can also be 
thought of as the result of the attachment of a generalized symplectic $1$-handle.
\end{abstract}
\begin{center}
\end{center}
%
%
\section*{Introduction}
%
%
Let $\Sigma$ denote a compact, $2n$-dimensional manifold admitting an exact symplectic form
$\omega = d\beta$ and let $Y$ denote the Liouville vector field defined by
$\iota_Y \omega = \beta$. Suppose that $Y$ is transverse to the boundary $\partial \Sigma$,
pointing outwards. These properties are precisely the ones requested
for $\Sigma$ to be a page of an abstract open book in the contact setting.
Given a symplectomorphism $\phi$ of $(\Sigma,\omega)$, equal to the identity 
near $\partial\Sigma$, one can, following a construction of Thurston 
and Winkelnkemper \cite{MR0375366} or rather its adaption to higher 
dimensions by Giroux \cite{MR1957051}, 
associate a $(2n+1)$-dimensional contact manifold 
$M_{(\Sigma,\omega,\phi)}$ to the data $(\Sigma,\omega,\phi)$.
\parskip 0pt

The main result of the present paper is part of the author's thesis
\cite{Klukas-Thesis}.
\begin{thm}
\label{thm:stein-monoid}
Given two symplectomorphisms $\phi_0$ and $\phi_1$ of $(\Sigma,\omega)$, 
equal to the identity near the boundary $\partial\Sigma$,  there
is an exact symplectic cobordism whose negative end equals the disjoint union of the contact manifolds $M_{(\Sigma,\omega,\phi_0)}$ and $M_{(\Sigma,\omega,\phi_1)}$, and whose positive end
equals $M_{(\Sigma,\omega,\phi_0 \circ \phi_1)}$. If, in addition, the page $(\Sigma,\omega)$ is a Weinstein manifold, then so is the cobordism.
\end{thm}
For $n=1$ the above statement is due to Baker--Etnyre--van Horn-Morris~\cite{2010arXiv1005.1978B} and, 
independently, Baldwin~\cite{MR2923174}. The general case of Theorem~\ref{thm:stein-monoid} was 
independently obtained by Avdek in \cite{2012arXiv1204.3145}, where the cobordism is associated with 
a so-called \textit{Liouville connected sum}. In \cite{Klukas-Thesis} I observed that the cobordism 
in Theorem~\ref{thm:stein-monoid} can also be understood as result of the attachment of a 
\textit{generalized symplectic $1$-handle} of the form $D^1 \times N(\Sigma)$, where $N(\Sigma)$ denotes a 
vertically invariant neighborhood of the symplectic hypersurface $\Sigma$. 
I will shed some more light on this in \S\ref{sec:1 handle}.
\parskip 0pt

From the methods introduced in the proof of Theorem~\ref{thm:stein-monoid} we can deduce some 
further applications such as the following.
We show the existence of strong fillings for contact manifolds associated with \textit{doubled open books}, 
a certain class of fibre bundles over the circle obtained 
by performing the binding sum of two open books with equal pages and inverse 
monodromies (cf.~\S\ref{sec:sym open books}). 
\begin{thm}
\label{thm:balanced fibration}
Any contact manifold associated to a doubled open book admits an exact symplectic filling.
\end{thm} 
In dimension $3$ the above statement appeared in \cite{MR3102479}, 
though details of the argument have not been carried out.
As outlined by Wendl in \cite{MR3102479}*{Remark 4.1} this statement has the following consequence 
for \textit{local filling obstructions}, i.e.\ subsets in the likes of overtwisted discs in dimension $3$ that inhibit the existence of a symplectic filling, in arbitrary dimensions.
Similar results in dimension $3$, concerning planar and Giroux torsion, 
are presented in \cite{MR3102479} and \cite{MR2737776}.  
\begin{cor}
\label{cor:filling obstr}
Let $(B,\pi)$ be an open book decomposition of a $(2n + 1)$-dimensional 
contact manifold $(M,\xi)$ and let $\mathcal{O} \subset (M,\xi)$  
be any local filling obstruction, 
then $B$ must intersect $\mathcal{O}$ non-trivially.
\end{cor}
Our final result will be the following. 
Let $(M,\xi)$ be a closed, oriented, $(2n+1)$-dimensional contact manifold supported by an open book with page $(\Sigma,\omega)$ and monodromy $\phi$.
Suppose further that $(\Sigma,\omega)$ symplectically embeds into a second $2n$-dimensional (not  necessarily  closed) symplectic manifold $(\Sigma',\omega')$, i.e.
\[
			(\Sigma,\omega) \subset (\Sigma',\omega').
\]
Let $M'$ be the symplectic fibration over the circle with fibre $(\Sigma',\omega')$ and monodromy equal to 
$\phi$ over $\Sigma \subset \Sigma'$ and equal to the identity elsewhere.
The following theorem has previously been proved by
D\"orner--Geiges--Zehmisch \cite{DGZ}. The proof in the present paper uses slightly different methods (cf.~\S\ref{sec:sympl fibration}).
\begin{thm}
\label{cor:symplectic fibration}
There is a smooth manifold $W$ with $\partial W = (-M) \sqcup M'$ and a symplectic 
form $\Omega$ on $W$ for which $(M,\xi)$ is a concave boundary component, 
and $\Omega$ induces $\omega'$ on the fibers of the fibration $M' \to S^1$. 
\end{thm}
For $n=1$ we could, for example, choose $\Sigma'$ to be the closed surface obtained by capping 
off the boundary components of $\Sigma$. Then Theorem~\ref{cor:symplectic fibration} would 
recover one of the main results (Theorem 1.1) in \cite{MR2023279}. The low-dimensional case ($n=1$) 
of Theorem~\ref{cor:symplectic fibration} was, using different methods, already carried out in \cite{MR3128981}.
One may think of Theorem~\ref{cor:symplectic fibration} as an extension of the result in \cite{MR2023279}, or \cite{MR3128981} respectively, to higher dimensions.
%
%
\subsection*{Acknowledgements}
%
The main results of the present paper are part of my thesis~\cite{Klukas-Thesis}. 
I want to thank my advisor Hansj\"org Geiges for introducing me to the world of contact topology. 
Furthermore I thank Chris Wendl for perceptive 
comments on an earlier version of the paper, in particular for suggesting to add 
Corollary~\ref{cor:filling obstr}.
Finally I thank Max D\"orner for bringing Theorem~\ref{cor:symplectic fibration} to my attention and 
for useful comments that helped to improve the exposition.

The author was supported by the DFG (German Research Foundation) as fellow of the 
graduate training program \textit{Global structures in geometry and analysis} 
at the Mathematics Department of the University of Cologne, Germany, and by grant GE 1245/2-1 to H. Geiges.
%
%
\section{Preliminaries}
%
%
\subsection{Symplectic cobordisms}
%
%
Suppose we are given a symplectic $2n$-manifold $(X,\omega)$, oriented by the volume form
$\omega^n$, such that the oriented boundary $\partial X$ decomposes as $\partial
X = (- M_-) \sqcup M_+$, where $-M_-$ stands for $M_-$ with reversed
orientation. Suppose further that in a neighborhood of $\partial X$ there is a
Liouville vector field $Y$ for $\omega$, transverse to the boundary and pointing
outwards along $M_+$, inwards along $M_-$. The $1$-form $\alpha = i_Y\omega$
restricts to $TM_\pm$ as a contact form defining cooriented contact structures
$\xi_\pm$.
\parskip 0pt

We call $(X,\omega)$ a \textbf{(strong) symplectic cobordism}
from $(M_-,\xi_-)$ to $(M_+,\xi_+)$, with \textbf{convex} boundary $M_+$ and
\textbf{concave} boundary $M_-$. In case $(M_-,\xi_-)$ is empty $(X,\omega)$ is called a
\textbf{(strong) symplectic filling} of $(M_+,\xi_+)$.
If the Liouville vector field is defined not only in a neighborhood of $\partial X$ but
everywhere on $X$ we call the cobordism or the filling respectively \textbf{exact}.
\parskip 0pt

A \textbf{Stein manifold} is an affine complex manifold, i.e.~a complex manifold
that admits a proper holomorphic embedding into $\C^N$ for some large integer $N$.
By work of Grauert \cite{MR0098847} a complex manifold $(X,J)$ is Stein if and only if it
admits an exhausting plurisubharmonic function $\rho\co X \to \R$.
Eliashberg and Gromov's symplectic
counterparts of Stein manifolds are \textit{Weinstein manifolds}.
\parskip 0pt

A \textbf{Weinstein manifold} is a quadruple $(X,\omega,Z,\varphi)$, see
\cite{eliGro-convex_symplectic_manifolds},  where $(X,\omega)$ is an exact
symplectic manifold, $Z$ is a complete globally defined Liouville vector field,
and $\varphi\co X \to \R$ is an exhausting (i.e.~proper and bounded below)
Morse function for which $Z$ is gradient-like.
Suppose $(X,\omega)$ is an exact symplectic cobordism with boundary $\partial X
= (- M_-) \sqcup M_+$ and with Liouville vector field $Z$. We call $(X,\omega)$
\textbf{Weinstein cobordism} if there exists a Morse function $\varphi\co X \to
\R$ which is constant on $M_-$ and on $M_+$, has no boundary critical points on $M_-$ and on $M_+$,
and for which $Z$ is gradient-like.
%
%
\subsection{Open books}
%
%
An \textbf{open book decomposition} of an $n$-dimensional manifold $M$ is a pair
$(B,\pi)$, where
  $B$ is a co-dimension $2$ submanifold in $M$, called the
  \textbf{binding} of the open book and
  $\pi\co M\setminus B \to S^1$ is a (smooth, locally trivial)
  fibration such that 
  each fibre $\pi^{-1}(\varphi)$, $\varphi\in S^1$, corresponds to the interior
  of a compact hypersurface $\Sigma_\varphi \subset M$ with
  $\partial\Sigma_\varphi = B$.
  The hypersurfaces $\Sigma_\varphi$, $\varphi \in S^1$, are called the
  \textbf{pages} of the open book.
\parskip 0pt

In some cases we are not interested in the exact position of the binding or the pages of an open book decomposition inside the ambient space. Therefore, given an open book decomposition $(B,\pi)$ of an $n$-manifold $M$, we could ask for the relevant data to remodel the ambient space $M$ and its underlying open books structure $(B,\pi)$, say up to diffemorphism. This leads us to the following notion.
\parskip 0pt

An \textbf{abstract open books} is a pair $(\Sigma,\phi)$, where $\Sigma$ is a compact hypersurface with non-empty boundary $\partial \Sigma$, called the \textbf{page} and $\phi\co\thinspace \Sigma \to \Sigma$ is a diffeomorphism equal to the identity near $\partial \Sigma$, called the \textbf{monodromy} of the open book.
Let $\Sigma(\phi)$ denote the mapping torus of $\phi$, that is, the quotient space obtained from $\Sigma \times [0,1]$ by identifying $(x,1)$ with $(\phi(x),0)$ for each $x \in \Sigma$. Then the pair $(\Sigma,\phi)$ determines a closed manifold $M_{(\Sigma,\phi)}$ defined by 
\begin{equation}
\label{eqn:abstract open book}
			M_{(\Sigma,\phi)} := \Sigma(\phi) \cup_{\id} (\partial \Sigma \times D^2),
\end{equation}
where we identify $\partial \Sigma(\phi) = \partial \Sigma \times S^1$ with $\partial (\partial \Sigma \times D^2)$ using the identity map.
Let $B \subset M_{(\Sigma,\phi)}$ denote the embedded submanifold $\partial \Sigma \times \{0\}$. Then we can define a fibration $\pi\co M_{(\Sigma,\phi)}\setminus B \to S^1$ by
\[
	\left.
	 \begin{array}{l}
				\lbrack x,\varphi \rbrack  \\
				 \lbrack \theta, r\e^{i\pi\varphi} \rbrack  
		\end{array} \right\}
							\mapsto [\varphi],
\]
where we understand $M_{(\Sigma,\phi)}\setminus B$ as decomposed as in Equation~\ref{eqn:abstract open
book} and $[x,\varphi] \in \Sigma(\phi)$ or $ [\theta, r\e^{i\pi\varphi}] \in
\partial\Sigma \times D^2 \subset \partial\Sigma \times \C$ respectively. Clearly $(B,\pi)$ defines an open book decomposition of $M_{(\Sigma,\phi)}$.
\parskip 0pt

On the other hand, an open book decomposition $(B,\pi)$ of some $n$-manifold
$M$ defines an abstract open book as follows: identify a neighborhood of $B$
with $B \times D^2$ such that $B = B\times\{0\}$ and such that the fibration on
this neighborhood is given by the angular coordinate, $\varphi$ say, on the
$D^2$-factor. We can define a $1$-form $\alpha$ on the complement $M \setminus (B \times
D^2)$ by pulling back $d\varphi$ under the fibration $\pi$, where this time we
understand $\varphi$ as the coordinate on the target space of $\pi$.
The vector field $\partial \varphi$ on $\partial\big(M \setminus (B \times D^2)
\big)$ extends to a nowhere vanishing vector field $X$ which we normalize by
demanding it to satisfy $\alpha(X)=1$. Let $\phi$ denote the time-$1$ map of the
flow of $X$. Then the pair $(\Sigma,\phi)$, with $\Sigma = \overline{\pi^{-1}(0)}$, defines an abstract open book such that$M_{(\Sigma,\phi)}$ is diffeomorphic to $M$.
%
%
\subsection{Compatibility}
\label{sec:compatibility}
%
%
Let $\Sigma$ denote a compact, $2n$-dimensional manifold admitting an exact symplectic form
$\omega = d\beta$ and let $Y$ denote the Liouville vector field defined by
$\iota_Y \omega = \beta$. Suppose that $Y$ is transverse to the boundary $\partial \Sigma$,
pointing outwards.
Given such a triple $(\Sigma,\omega,\phi)$ a construction of Giroux 
\cite{MR1957051}, cf.\ also \cite{MR2397738}*{\S 7.3}, 
produces a contact manifold $M_{(\Sigma,\omega,\phi)}$ whose contact structure is 
adapted to the open book in the following sense.

A positive contact structure $\xi = \ker \alpha$ and an open book decomposition $(B,\pi)$ of an $(2n + 1)$-dimensional manifold $M$ are said to be \textbf{compatible}, if
the $2$-form $d\alpha$ induces a symplectic form on the interior $\pi^{-1}(\varphi)$ of each page, defining
its positive orientation, and the $1$-form $\alpha$ induces a positive contact form on $B$.
%
%
\section{Concatenation of open books and symplectic fibrations}
\label{sec:concatenation}
\label{sec:sympl fibration}
%
\newcommand{\heightOfSigma}{\boldsymbol{r}}
The following definitions will turn up in the proofs of all our main results: 
let $\Sigma$ denote a compact, $2n$-dimensional manifold admitting an exact symplectic form
$\omega = d\beta$ and let $Y$ denote the Liouville vector field defined by
$\iota_Y \omega = \beta$. Suppose that $Y$ is transverse to the boundary $\partial \Sigma$,
pointing outwards. Denote by $(r,x)$ coordinates on a collar neighborhood 
$(-\varepsilon,0] \times \partial \Sigma$ induced by the negative 
flow corresponding to the Liouville vector field $Y$.
Let $(\hat{\Sigma},\hat\omega = d\hat{\beta})$ denote the completion of $(\Sigma, \omega)$, 
obtained by attaching the positive half 
$\big([0,\infty) \times \partial\Sigma, d(\text{e}^r\beta_0) \big)$ of the 
symplectization of $(\partial\Sigma, \beta_0 = \beta|_{\partial\Sigma})$.
The Liouville vector field $Y$ extends over $\hat\Sigma$ by $\partial_r$ and we will continue 
to denote the extended vector field by $Y$.
Let
$
	\heightOfSigma \co \hat\Sigma \to \mathbb{R}_{\geq 0}
$
be a smooth function on $\hat\Sigma$, 
satisfying the following properties:
\begin{list}{$\bullet$}{}
  \item $\heightOfSigma \equiv 0$ over $\hat\Sigma \setminus \big( (-\varepsilon , \infty ) \times \partial\Sigma \big)$,
  \item $\frac{\partial\heightOfSigma}{\partial_r} > 0$ and $\frac{\partial\heightOfSigma}{\partial_x} \equiv 0$ over $\big( (-\varepsilon , \infty ) \times \partial\Sigma \big)$ with coordinates $(r,x)$, and
	 \item $\heightOfSigma(r,x) = r + 1$ over $\big( [0 , \infty ) \times \partial\Sigma \big)$.
\end{list}
Note that over the collar neighborhood $(-\varepsilon,\infty) \times \partial \Sigma$ the vector field $Y$ is 
gradient-like for $\heightOfSigma$.

In order to define the desired Liouville vector fields in our proofs we will need the following ingredient.
For some sufficiently small $\delta >0$ let
$g \colon\thinspace [0,\delta] \to \R$ be a functions satisfying the
following properties:
\begin{list}{$\bullet$}{}
  	\item $g(y) = 1$, for $y$ near $0$,
  	\item $g(y) = 0$, for $y$ near $\delta$,
  	\item $g'(y) \leq 0$, for each $y \in [0,\delta]$.
\end{list}
We are now ready to construct the desired exact symplectic cobordisms of Theorem~\ref{thm:stein-monoid} as well as of 
Theorem~\ref{cor:symplectic fibration}.
%
%
\begin{proof}[Proof of Theorem \ref{thm:stein-monoid}]
%
%
The starting point for the desired cobordism will be 
the space $\hat\Sigma \times \R^2$ with coordinates $(p, x,y)$. This space is symplectic with symplectic form
\[
	 \Omega =   \hat\omega + dx\wedge dy .
\]
Let $P = P_{a,b,c}$ denote the subset of $\hat\Sigma \times \R^2$ defined by
\[
 			P := \big\{(p,x,y)\co  \heightOfSigma \leq 0, \  x^2 + y^2 \leq c^2 \  \mbox{and} \ (x
 			\pm b)^2 + y^2 \geq  a^2
 			\big\},
\]
where $a,b,c \in \R$ are some potentially very large constants satisfying $a<b<b+a<c$.
The final choice of these constants will later ensure that our 
desired (and yet to be defined) Liouville vector field $Z$ will be transverse to the boundary of the cobordism.
Consider the vector field $Z'=Z'_{b}$ on $\hat\Sigma \times \R^2$ defined by
\begin{equation}
\label{eqn:Z prime}
	Z' = Y + X, 
\end{equation}
where $X = \big(1 - f'(x)\big)y\thinspace\partial_y +
f(x)\thinspace\partial_x$ and  $f:\R \to \R$ is a function satisfying the
following properties:
\begin{list}{$\bullet$}{}
 	\item $f(\pm b) = f(0)=0$,
 	\item $f'$ has exactly two zeros $\pm x_0$ with $0 < |x_0| < b$,
 	\item $|f'(x)| < 1$ for each  $x\in \R$, and 
 	\item $\lim_{{x \to \pm \infty}}f(x) =  \pm \infty.$ 
\end{list}
An easy computation shows that $X$ is a Liouville vector field on
$(\R^2,dx\wedge dy)$ for any function $f$. Hence $Z'$ defines a Liouville
vector field on $\big(\hat\Sigma \times \R^2,\Omega \big)$. Note that $Z'$ is transverse to the boundary of $P$, cf. Figure~\ref{fig:vf X}.
\begin{figure}[h] 
\center
\includegraphics{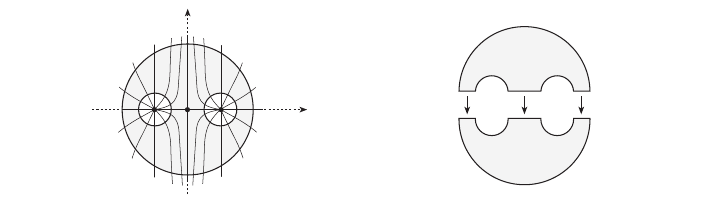}
\put (-193,38){$x$}
\put (-248,88){$y$}
\put (-93,72){$P_+$}
\put (-93,22){$P_-$}
\put (-132,46){$\phi_0$}
\put (-56,46){$\phi_1^{-1}$}
\caption{Left: Flow lines of the Liouville vector field $X$. Right: Construction of $P(\phi_0,\phi_1)$}
\label{fig:vf X}
\end{figure}
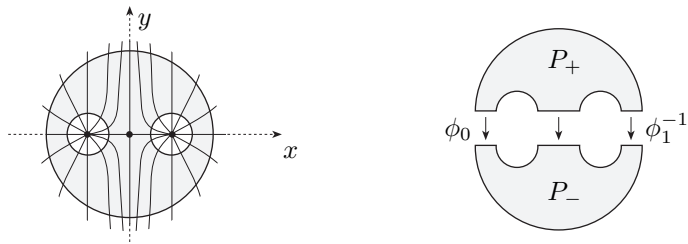
\parskip 0pt

We are now ready to define the desired exact symplectic cobordism $(W, \Omega ,Z)$.
Start by cutting $P$ along $\{y=0\}$ and re-glue
with respect to $\phi_0$ and $\phi_1$ as follows. Set
$P_\pm := P \cap \{ \pm y \geq 0 \}$ and $P_0 = P \cap \{y=0\}$. 
Obviously $P_0$ can be understood as part of the boundary of $P_+$ as well as of $P_-$.
Now consider
\[
			P(\phi_0,\phi_1) :=  ( P_+\sqcup P_- ) /_{\sim_\Phi},
\]
where we identify with respect to the map $\Phi\co P_0  \to  P_0$ given by 
\[
	\Phi( p, x,0) :=
		\begin{cases}
  				( \phi_0(p), x,0 ) & \text{, for $x < - b$,}\\
					( 				p, x,0 ) & \text{, for $|x| < b$,} \\ 
  				( \phi_1^{-1}(p), x,0 ) & \text{, for $b < x$.} 
  				
		\end{cases}
\]
Here we understand the domain of definition of $\Phi$ as part of the boundary of $P_+$ 
and the target space as part of $P_-$, cf. also Figure~\ref{fig:vf X}.
Note that, since $\phi_0$ and $\phi_1$ are
symplectomorphisms of $(\Sigma,\omega)$, and can be assumed to equal the identity over $(-\varepsilon,0]\times \partial\Sigma$, they extend trivially over $\hat\Sigma$. Furthermore $\Phi$ keeps the $x$-coordinates
fixed, and hence $\Omega$ descends to a symplectic
form on $P(\phi_0,\phi_1)$ which we will continue to denote by $\Omega$.
We are now going to define the Liouville vector field $Z$ on $P(\phi_0,\phi_1)$.
Without any loss of generality 
the symplectomorphisms $\phi_0$ and $\phi_1^{-1}$ can be chosen to be exact (cf.~\cite{MR2397738}), 
i.e.~we have $\phi_0^*\hat\beta - \hat\beta = d\varphi_0$ and $(\phi_1^{-1})^*\hat\beta - \hat\beta = d\varphi_1$ 
defining functions $\varphi_0$ and $\varphi_1$ on $\Sigma$, unique up to
adding a constant. Hence we may assume that $\varphi_0$ and $\varphi_1$ vanish over $(-\varepsilon,\infty) \times \partial \Sigma$.
To avoid confusing
indices we will write
\[
			\Phi^*\hat\beta - \hat\beta = d\varphi
\]
to summarize these facts. Let $g \colon\thinspace [0, \varepsilon] \to \R$ be the function 
as defined at the beginning of the present section. 
Over $P_+$ we define $Z = Z_b$ to be given as
\begin{equation}
\label{eqn:Z}
			Z = \Big( g(y)\thinspace (T\Phi^{-1})(Y) +  \big(1-g(y) \big) \thinspace Y \Big)+
			 			X + g'(y)\varphi(p)\thinspace\partial_x.
\end{equation}
To show that $Z$ is indeed a Liouville vector field we have to take a look
at the Lie derivative of $\Omega$ along $Z$. With the help of the Cartan
formula we compute
\begin{eqnarray*}
			\mathcal{L}_{Z}\Omega 
		&=& 
			d \big( g\thinspace \Phi^*\hat\beta + (1-g) \thinspace \hat\beta \big) 
						+ dx \wedge dy 
						+ d(g'\varphi\thinspace dy) \\
		&=&  
			\big( dg\wedge(\Phi^*\hat\beta) - dg \wedge \hat\beta +
			g\thinspace(\Phi^*\hat\omega) + (1-g)\thinspace\hat\omega \big)  + dx \wedge dy  +
			g'\thinspace d\varphi\wedge dy  \\
		&=&
			\big( g'\thinspace dy\wedge(\Phi^*\hat\beta) - g' \thinspace dy \wedge \hat\beta +
			g\thinspace \hat\omega + (1-g)\thinspace\hat\omega \big)  + dx \wedge dy  +
			g'\thinspace d\varphi\wedge dy  \\
		&=& 
			\big( g'\thinspace dy\wedge d\varphi+
			\hat\omega \big)  + dx \wedge dy  -
			g'\thinspace dy \wedge d\varphi  \\
		&=& 
			\hat\omega  + dx \wedge dy \\ &=& \Omega.
\end{eqnarray*}
Observe that, since $T\Phi(Z|_{P_0}) = Z'|_{P_0}$, we can extend $Z$ over $P_-$ by
$Z'$. In particular $Z$ descends to a vector field on $P(\phi_0,\phi_1)$.
Let $W' = W'_{a,b,c}$ denote the subset of $\hat\Sigma \times \R^2$ defined by
\[
			W' := \big\{ (p,x,y)\ | \ 
				\heightOfSigma^2 + (x \pm b)^2 + y^2 \geq a^2
				\ \ \mbox{and} \ \ 
				\heightOfSigma^2 + x^2  + y^2  \leq c^2 
				\big\}
\]
and note that we have $P \subset W'$. Finally we define the symplectic
cobordism $W = W_{a,b,c}$ by
\[
			W := ( W'\setminus P ) \cup P(\phi_0,\phi_1).
\]
The boundary of $W$ decomposes as $\partial W = \partial_-W
\sqcup \partial_+W$, where we have
\begin{equation}
\label{eqn:boundary of W}
			\partial_-W =  \heightOfSigma^2 + (x \pm b)^2 + y^2 = a^2 \} \ \ \mbox{and} \ \
			 \partial_+W = \{ \heightOfSigma^2 + x^2  + y^2  = c^2  \} .
\end{equation}
We do not have to worry about the well-definedness of the function 
$\boldsymbol{r}$ on $P(\phi_0,\phi_1) \subset W$ since $\phi_0$ and $\phi_1$ can be assumed to equal the
identity over $(-\varepsilon, \infty) \times \partial\Sigma$, which is the only region where $\heightOfSigma$ is
non-trivial.

Observe that yet we cannot fully ensure that 
the Liouville vector field $Z$ is transverse to $\partial W$ pointing inwards
along $\partial_-W$ and outwards along $\partial_+W$. However the only problem is the 
last term in $Z$, namely the term $g'(y)\varphi(p)\thinspace\partial_x$. 
We tame this deviation as follows:
up to this point we have not fixed the constants $a,b,c$ yet. By choosing $a,b,c$ sufficiently
large the deviation induced by $g'(y)\varphi(p)\thinspace\partial_x$ becomes non-essential and 
the Liouville vector field $Z$ becomes transverse to $\partial W$ pointing inwards
along $\partial_-W$ and outwards along $\partial_+W$.

It remains to show that 
we indeed have $\partial_-W = M_{(\Sigma,\omega,\phi_0)} \sqcup
M_{(\Sigma,\omega,\phi_1)}$ and $\partial_+W = M_{(\Sigma,\omega,\phi_0 \circ
\phi_1)}$. 
We start with the negative boundary components of $W$. Denote by $(\partial_-W)_{\pm} \subset  \partial_-W $ the two distinct components of the negative boundary of $W$ (cf. Equation~\ref{eqn:boundary of W}).
Set 
\[
		B_\pm := \{  x = \pm  b, y = 0 \} \subset (\partial_-W)_{\pm}
\] 
and note that $B_\pm$ has trivial normal bundle. Further note that 
the complement $(\partial_-W)_{\pm}\setminus B_\pm$ admits a fibration over the circle defined by  
\[
	\pi_\pm (p,x,y) :=  \frac{(x \mp b,y)}{ \| (x\mp b,y) \| } \in S^1.
\]
This definition is compatible with the gluing induced by $\Phi$ and 
defines an open book decomposition of $(\partial_-W)_{\pm}$ with pages diffeomorphic to $\Sigma$ and 
monodromy given by $\phi_0$ and $\phi_1$ respectively. In the definition of $X$ (cf. Equation~\ref{eqn:Z prime}) we can choose
the underlying function $f$ such that in a neighborhood of $\pm b$ it is given by $f(x) = \frac{1}{2}x$. 
Therefore in a neighborhood of the binding $B_\pm \subset (\partial_-W)_{\pm}$ the $1$-form $\iota_Y \Omega$ is given by
\begin{equation}
\label{eqn:induced form is standard}
			\iota_Z \Omega = \beta + \tfrac{1}{2}\big(x \thinspace dy - y \thinspace dx\big).
\end{equation}
In addition, pulling back $d (\iota_Z \Omega) = \Omega$ to a fiber $(\pi_\pm)^{-1}(\theta)$, $\theta \in S^1$, yields the given symplectic form $\omega$.
The projection on the $p$-coordinate yields a symplectomorphism of each page $(\pi_\pm)^{-1}(\theta)$ endowed with the symplectic form induced by $\Omega$ to the subset $\Sigma_{\heightOfSigma < a} \subset (\hat\Sigma, \hat\omega)$.
Note that, since we are just interested in the induced contact structure on $\partial_-W$ (not the whole cobordism $W$), we are allowed to set $a=1$.
This shows that the contact structure on $(\partial_-W)_{\pm}$ induced by $\iota_Z \Omega$ is the same as given by the generalized Thurston-Winkelnkemper construction (cf.~\cite{MaxThesis}).

The argument for $\partial_+W$ is almost similar -- except for the fact that in a neighborhood of the binding the vector field $X$ (cf. Equation~\ref{eqn:Z prime}), or rather the underlying function $f$, is not of the right form. However since we are just interested in the induced contact structure on $\partial_+W$ (not the whole cobordism $W$) we are allowed to change $f$ accordingly: choose an isotopy $(f_t)_{t\in[0,1]}$ with $f_0\equiv f$, $f_1(x)=\frac{1}{2}x$ and such that for each $f_t$ the induced vector field $Z$ (cf. Equation~\ref{eqn:Z}) stays transverse to $\partial_+W$. The we obtain a contact structure for each $t \in [0,1]$ on $\partial_+W$ all of which are contactomorphic by Gray stability. Set
\[
		B := \{  x = y = 0 \} \subset \partial_+W
\] 
and note that $B$ has trivial normal bundle. Further note that 
the complement $( \partial_+W )\setminus B$ admits a fibration over the circle defined by  
\[
	\pi(p,x,y) :=  \frac{(x,y)}{ \| (x,y) \| } \in S^1.
\]
Following the same line of arguments as for $\partial_-W$ one concludes that the contact structure on $\partial_+W$ induced by $\iota_Z \Omega$ is the same as given by the generalized Thurston-Winkelnkemper construction (cf.~\cite{MaxThesis}).
\begin{figure}[h] 
\center
\includegraphics{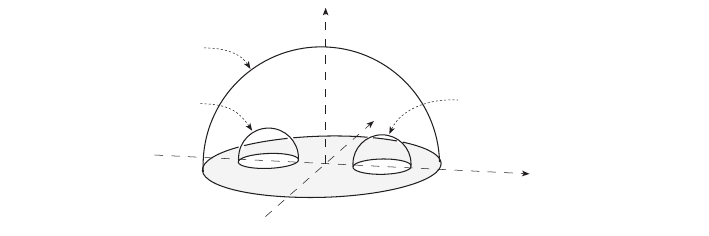}
\put(-180,100){$\heightOfSigma$}
\put(-95,18){$x$}
\put(-171,55){$y$}
\put(-117,61){$M_{(\Sigma,\omega,\phi_1)}$}
\put(-287,61){$M_{(\Sigma,\omega,\phi_0)}$}
\put(-300,92){$M_{(\Sigma,\omega,\phi_0 \circ \phi_1)}$}
\caption{Schematic picture of the symplectic cobordism constructed in Theorem
\ref{thm:stein-monoid}.}
\label{fig:the cobordism}
\end{figure}
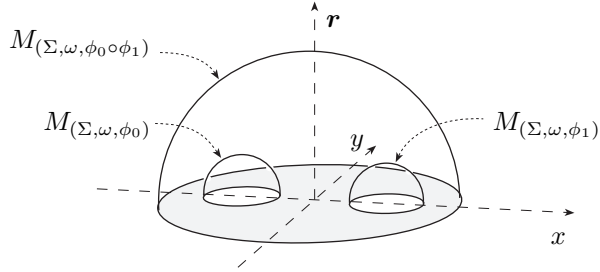
\end{proof}
As mentioned above we will now briefly sketch an alternative approach to Theorem~\ref{thm:stein-monoid} utilizing a 
\textit{generalized symplectic $1$-handle} as defined in \S\ref{sec:1 handle}.
A similar approach was independently followed by Avdek in \cite{2012arXiv1204.3145}.
%
%
\begin{proof}[Sketch of the alternative approach]
%
%
Suppose we are given two $(2n +1 )$-dimensional contact manifolds $(M_0,\xi_0)$ and $(M_1,\xi_1)$.
Suppose further that they are associated with compatible open books $(\Sigma, \omega, \phi_0)$ 
and $(\Sigma, \omega, \phi_1)$ with equal pages.
For $i=0,1$ let $\pi_i\colon\thinspace M_i \setminus B_i \to S^1$ denote the induced fibrations.

Note that the subsets $ \pi_i^{-1}\big((-\varepsilon,\varepsilon)\big) \setminus (B_i \times D^2_\varepsilon) \subset (M_i,\xi_i)$, $i=0,1$, 
define an embedding
\[
	 S^0 \times N_\varepsilon(\Sigma) \hookrightarrow M_0 \sqcup M_1,
\]
where $N_\varepsilon(\Sigma)$ denotes a neighborhood of $\Sigma$ as described in \S\ref{sec:1 handle}.
We can understand this as the attaching region $\mathcal{N}$ of a 
generalized $1$-handle $H_\Sigma$ as described in \S\ref{sec:1 handle}.
Attaching $H_\Sigma$ to the positive end of a symplectization of $M_0 \sqcup M_1$ one can show that 
we end up with a cobordism whose positive end equals the contact manifold associated to
 $(\Sigma, \omega, \phi_0 \circ \phi_1)$, cf.\ Figure~\ref{fig:concatenation}.
In the proof of Theorem~\ref{thm:stein-monoid} the subset $W_{|x|<b} \subset W$ can be understood 
as a perturbed instance of a handle $H_\Sigma$. 
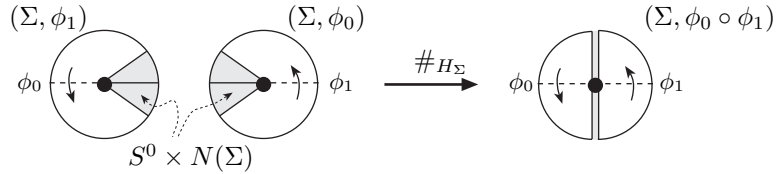
\begin{figure}[h] 
\center
\includegraphics{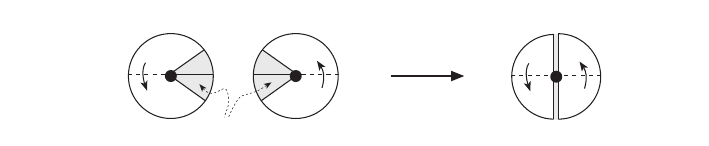}
\put (-190,56){$(\Sigma, \phi_0)$} 
\put (-295,56){$(\Sigma, \phi_1)$}
\put (-292,32){\small{$\phi_0$}}
\put (-174,32){\small{$\phi_1$}}
\put (-55,56){$(\Sigma, \phi_0 \circ \phi_1)$}
\put (-106,32){\small{$\phi_0$}}
\put (-49,32){\small{$\phi_1$}}
\put (-250,3){$S^0 \times N(\Sigma)$}
\put (-142,40){$\#_{H_\Sigma}$}
\caption{Schematic picture of summing two open books along thickened
pages.}
\label{fig:concatenation}
\end{figure}
\end{proof} 
Note that, just as in the previous case, it is possible to 
obtain the cobordism in Theorem~\ref{cor:symplectic fibration} 
by the attachment of
a \textit{generalized symplectic $1$-handle} (cf.\ \S\ref{sec:1 handle})
to $\Sigma' \times D^2$ with symplectic form $\omega' + dx \wedge dy$ 
and to the symplectization of $(M,\xi)$. However in the present 
paper we provide a direct construction.
%
%
\begin{proof}[Proof of Theorem~\ref{cor:symplectic fibration}]
%
%
Let $\heightOfSigma|_\Sigma$ be the restriction of a $C^\infty$-function on 
$\Sigma$ as defined at the beginning of \S\ref{sec:concatenation}.
Let 
\[
    \boldsymbol r \co\Sigma' \to \mathbb{R}_{\geq 0}
\]
denote a smooth extension of this function to the rest of $\Sigma'$ such that on $\Sigma' \setminus \Sigma$ 
we have $\heightOfSigma > 1$.
Note that the Liouville vector field $Y$ on $\Sigma \subset \Sigma'$ associated to the primitive $\beta$ of $\omega$ can be slightly extended
into $\Sigma'\setminus \Sigma$.
In analogy to the proof of Theorem \ref{thm:stein-monoid} we 
consider the symplectic space $\Sigma' \times \R^2$ with symplectic 
form
\[
	\Omega = \omega' + dx \wedge dy.
\]
Over $\Sigma \times \R^2 \subset \Sigma' \times \R^2$ we define the Liouville 
vector field 
\[
	Z' = Y +  \tfrac{1}{2}\big(x\thinspace\partial_x + y\thinspace\partial_y \big).
\]
The final cobordism will depend on some potentially very large constants $a,b \in \mathbb{R}$ satisfying
$0 < a < b$.
These constants will be fixed later to ensure the Liouville vector field is transverse to the boundary.
Let $A = A_{a,b}$ denote the subset of $\Sigma' \times \R^2$ defined by
\[
			A := \big\{ (p,x,y) \co \heightOfSigma \leq 0  
				\  \mbox{and}  \ 
				  a^2 \leq x^2 + y^2 \leq b^2 \big\}.
\]
In analogy of the definition of $P(\phi_0,\phi_1)$ in the proof of 
Theorem~\ref{thm:stein-monoid} we define $A(\phi)$ as follows.
Set $A_\pm := A \cap \{ \pm y \geq 0 \}$ and $A_0 = A \cap \{ y=0 \}$. 
We can understand $A_0$ as part of the boundary of $A_+$ as well as of $A_-$. 
We define
\[
          A(\phi) := (A_+ \sqcup A_-) /_\sim, 
\]
where we identify with respect to the map $\Phi\co A_0 \to A_0$ given by
\[
	\Phi( p, x,0) :=
		\begin{cases}
  				( \phi(p); x,0 ) & \text{, for $x < 0$,}\\
  				( 				p; x,0 ) & \text{, for $x > 0$.} 
		\end{cases}
\]
Let $W' = W'_{a,b}$ denote the subset of $\Sigma \times \R^2$ defined by
\[
			W' := \big\{ (p,x,y) \co \heightOfSigma^2 + x^2 + y^2 \geq {a^2} 
				\ \ \mbox{and} \ \ 
				x^2 + y^2 \leq b^2 \big\}
\]
and note that we have $A \subset W'$. Finally we define the symplectic
cobordism $W$ by
\[
			W := ( W'\setminus A ) \cup A(\phi).
\]
Observe that $\Omega$ descends to a symplectic form on $W$. Furthermore we
indeed have $\partial W = (-M) \cup M'$.

It remains to define the desired Liouville vector field. This will be done in analogy to 
the construction in the proof of Theorem~\ref{thm:stein-monoid}.
The symplectomorphisms $\phi$ can be chosen to be exact (cf.~\cite{MR2397738}), 
i.e.~we have $\phi^*\beta - \beta = d\varphi$ defining a function $\varphi$ on $\Sigma$, unique up to
adding a constant. Hence we may assume that $\varphi$ vanish over a neighborhood of $\partial \Sigma$.
Let $g\co [0,\varepsilon] \to \R$ be a function as defined at the beginning of \S\ref{sec:concatenation}. 
Choose a sufficiently small $\delta>0$ such that $Y$ is still defined over $\Sigma'_{\boldsymbol r < 1+\delta}$.
We define the Liouville vector 
field $Z$ on $W_{\heightOfSigma \leq 1 + \delta}$ by
\[
			Z = \Big( g(y)\thinspace (T\Phi^{-1})(Y) +  (1-g(y)) \thinspace Y \Big)
			+ \tfrac{1}{2}\big(x\thinspace\partial_x + y\thinspace\partial_y \big)
			 + g'(y)\varphi(p)\thinspace\partial_x.
\]
As in the proof of Theorem~\ref{thm:stein-monoid} for sufficiently large 
constants $a,b$ the Liouville vector field $Z$ is 
transverse to the lower boundary $\partial_-W = M_{(\Sigma,\omega,\phi)}$ 
pointing inwards. Finally observe that $\Omega$ induces $\omega'$ on the 
fibers of the fibration $M' \to S^1$, which completes the proof.
\end{proof}
\begin{figure}[h] 
\center
\includegraphics{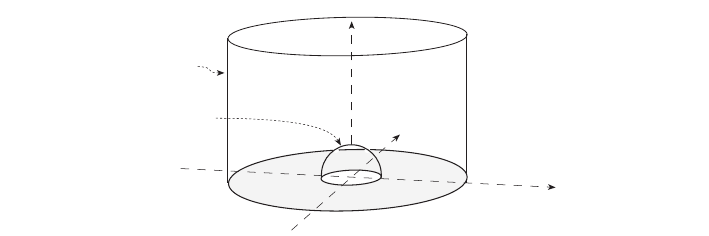}
\put(-182,98){$\heightOfSigma$}
\put(-87,13){$x$}
\put(-147,43){$y$}
\put(-255,55){$M$}
\put(-260,80){$M'$}
\caption{Schematic picture of the symplectic cobordism constructed in Corollary
\ref{cor:symplectic fibration}.}
\label{fig:symplectic fibration}
\end{figure}
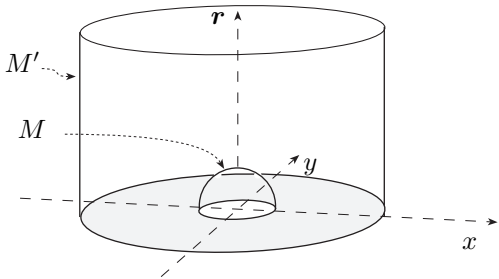
%
%
\section{Exact fillings of doubled open books}
\label{sec:sym open books}
%
%
In the present section we show the strong fillability of contact manifolds obtained by a 
certain doubling construction. In short, we perform the binding sum of 
two open books with equal pages and inverse 
monodromies. To be more precise we do the following:
let $(M_0,\xi_0)$ be a contact $(2n+1)$-manifold supported by an open book 
$(\Sigma,\omega,\phi)$ and let $(M_1,\xi_1)$ be the contact manifold associated to 
the open book $(\Sigma,\omega,\phi^{-1})$. Denoting by $B$ the boundary of $\Sigma$ 
we can form a new contact manifold $(M',\xi')$ as follows: 
the binding $B$ defines a codimension-$2$ contact submanifold $B_i \subset (M_i,\xi_i)$ for $i=0,1$.
Their normal bundles $\nu B_0$ and $\nu B_1$ admit trivializations induced by the pages 
of the respective open book decompositions of $M_0$ and $M_1$. Hence, 
we can perform the fiber connected sum, cf.~\cite{MR2397738}*{\S 7.4}, along each copy 
of $B$ with respect to these trivializations of the normal bundles and denote the 
result by $(M',\xi')$, i.e.\ denoting by $\Psi$ the fibre orientation reversing 
diffeomorphism of $B \times D^2 \subset B \times \C$ sending $(b,z)$ to $(b,\bar z)$, 
using the notation in \cite{MR2397738}*{\S 7.4}, we define
\[
		(M',\xi') := (M_0,\xi_0)\#_\Psi (M_1,\xi_1).
\]
The result $(M',\xi')$ defines a fibration over the circle with fibre given by 
\[
			\Sigma' = (-\Sigma) \cup_B \Sigma.
\]
Note that each fibre $\Sigma'$ defines a convex hypersurface, i.e.\
there is a contact vector field $X$ on $(M',\xi')$ which is transverse to the fibers. 
Furthermore for each fibre $\Sigma'$ the contact vector field $X$ is tangent to the contact structure exactly over $B$. 
We will refer to $(M',\xi')$ as a \textbf{doubled open book} and will sometimes denote it by
\[ 
	(\Sigma,\phi)\boxplus (\Sigma,\phi^{-1}).
\]
Before we show the existence of the desired symplectic filling in the above statement,
we show how it can be utilized to prove Proposition~\ref{cor:filling obstr}.
We follow the outline presented in \cite{MR3102479}*{Remark 4.1}. 
\begin{proof}[Proof of Corollary \ref{cor:filling obstr}]
Let $(M,\xi)$ be a contact manifold with a compatible open book decomposition $(B,\pi)$. Choose an 
arbitrarily small neighborhood of the binding $N_B \subset (M,\xi)$ and let  
$(M',\xi')$ denote the doubled open book associated to $(B,\pi)$. Obviously we can understand 
$(M,\xi) \setminus N_B$ as embedded in $(M',\xi')$. Since by Theorem~\ref{thm:balanced fibration} the doubled open book 
$(M',\xi')$ admits an exact filling we conclude that $(M,\xi) \setminus N_B$ cannot contain any local filling obstruction.
\end{proof}
It remains to show the existence of a symplectic filling for any doubled open book. 
%
%
\begin{proof}[Proof of Theorem \ref{thm:balanced fibration}]
%
%
Let $\heightOfSigma \co \hat\Sigma \to [0,\infty)$ be a $C^\infty$-function on the completion $(\hat\Sigma,\hat\omega)$ of $(\Sigma,\omega)$ as defined at the beginning of \S\ref{sec:concatenation}. 
Consider the symplectic space $\hat\Sigma \times \R \times [0, 2\pi]$ with coordinates $(p,x,y)$ and symplectic form $\Omega = \hat\omega + dx\wedge dy$. Set
\[
					\hat A := ( \hat\Sigma \times \R \times [0, 2\pi] ) /_\sim,
\] 
where we identify with respect to the map defined by $\Phi \co (p,x,0) \mapsto (\phi^{-1}(p),x, 2\pi)$. 
Since $\phi$ is a symplectomorphism of $(\Sigma,\omega)$, equal to the identity in a sufficiently small neighborhood of the boundary $\partial \Sigma$, the symplectic form $\Omega$ on $\hat\Sigma \times \R \times [0, 2\pi]$ descends to a symplectic form on $\hat A$ which we continue to denote by $\Omega$.
Note that the fibers of the projection $\hat A \to \R$ on the $x$-coordinate are diffeomorphic to the mapping torus $\Sigma(\phi)$.
\parskip 0pt

Let $W = W_a \subset \hat A$, for some constant $a > 0$, denote the subset defined by 
\[
				W := \{ (p,x,y) \co \heightOfSigma^2 + x^2 \leq a^2 \}.
\]
We follow the line of reasoning in the proof of Theorem \ref{thm:stein-monoid}.
The symplectomorphisms $\phi$ can be chosen to be exact (cf.~\cite{MR2397738}), 
i.e.~we have $\phi^*\beta - \beta = d\varphi$ defining a function $\varphi$ on $\Sigma$, unique up to
adding a constant. Hence we may assume that $\varphi$ vanish over a neighborhood of $\partial \Sigma$.
With $g\co [0, \varepsilon] \to \R$ as defined at the beginning of \S\ref{sec:concatenation} 
we define a Liouville vector field $Z$ on $W$ by
\[
			Z = \big( g(y)\thinspace (T\Phi^{-1})(Y) +  (1 -  g(y)) \thinspace Y \big)
				+ x\thinspace\partial_x
			 + g'(y)\varphi(y)\thinspace\partial_x.
\]
For sufficiently large $a > 0$ this vector field is transverse to the boundary $\partial  W_a$ 
of the subset $W_a$ pointing outwards.
Observe that the fibres of the projection $\partial W_a \to \R$ on the $y$-coordinate are diffeomorphic to $\Sigma' = (-\Sigma) \cup_B \Sigma$.
Moreover we have $\partial W_a = M'$. Finally observe that 
$Z$ indeed induces the contact structure $\xi'$. 
\end{proof}
For the sake of completeness we briefly sketch how the above symplectic cobordism
can be obtained by the attachment of a \textit{generalized symplectic $1$-handle} 
as defined in \S\ref{sec:1 handle}.
A similar approach was independently followed by Avdek in \cite{2012arXiv1204.3145}.
%
%
\begin{proof}[Sketch of the alternative approach]
%
%
Appearing as as convex boundary of $(\Sigma \times D^2, \omega + dx \wedge dy)$ the contact manifold 
$M_{(\Sigma,\omega,\tiny\id)}$ associated to a trivial 
open book $(\Sigma,\omega,\id)$ obviously admits a symplectic filling.
Recall that $M_{(\Sigma,\tiny\id)}$ is given by $(\Sigma \times S^1) \cup_{\id}(\partial \Sigma \times D^2)$.
For any symplectomorphism $\phi$ of $(\Sigma,\omega)$, equal to the identity near
$\partial \Sigma$, the part $(\Sigma \times S^1)$ can be described as
\[
	(\Sigma \times S^1) \cong 	\Big( \Sigma \times [0,1] \sqcup \Sigma \times [2,3] \Big) /_\sim,
\]
where we identify $(x,3)$ with $(\phi(x),0)$ and $(\phi(x),1))$ with $(x,2))$ for all $x \in \Sigma$.
Consider the subsets $\Sigma \times (\tfrac{1}{2} - \varepsilon,\tfrac{1}{2} + \varepsilon )$ and
$\Sigma \times (\tfrac{5}{2} - \varepsilon,\tfrac{5}{2} + \varepsilon )$ of 
$(\Sigma \times S^1) \subset M_{(\Sigma,\tiny\id)}$. They
define an embedding
\[
	S^0 \times  N(\Sigma) \hookrightarrow M_{(\Sigma,\tiny\id)},
\]
where $N(\Sigma)$ denotes a neighborhood of $\Sigma$ as described in \S\ref{sec:1 handle}.
We can understand this as the attaching region $\mathcal{N}$ of a 
generalized $1$-handle $H_\Sigma$ as described in \S\ref{sec:1 handle}.
Attaching $H_\Sigma$ to the positive boundary of a symplectic filling of $M_{(\Sigma,\omega,\tiny\id)}$
we end up with a cobordism whose positive end is easily identified as the contact manifold associated to
the doubled open book $(\Sigma,\phi)\boxplus (\Sigma,\phi^{-1})$, cf.\ Figure~\ref{fig:sym ob}.
\end{proof}
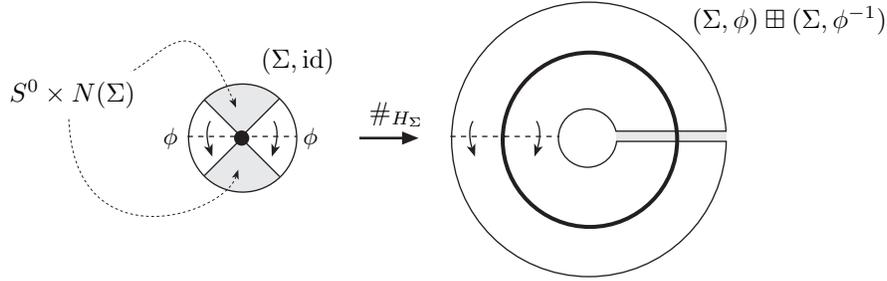
\begin{figure}[h] 
\center
\includegraphics{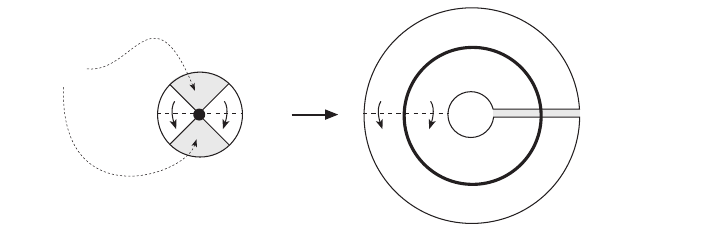}
\put (-238,85){$(\Sigma, \id)$}
\put (-75,100){$(\Sigma,\phi)\boxplus (\Sigma,\phi^{-1})$}
\put (-275,56){\small{$\phi$}}
\put (-222,56){\small{$\phi$}}
\put (-336,72){\colorbox{white}{$S^0 \times N(\Sigma)$} }
\put (-197,65){$\#_{H_\Sigma}$}
\caption{The attachment of a generalized $1$-handle yields a doubled open book.}
\label{fig:sym ob}
\end{figure}
%
%
%
\section{A generalised symplectic $1$-handle}
\label{sec:1 handle}
%
\newcommand{\homeOfHandle}{\hat\Sigma \times \mathbb{R}^2}
We assume that the reader is familiar with the idea behind the symplectic handle constructions 
due to Eliashberg \cite{MR1044658} and Weinstein \cite{MR1114405}. For an introduction we 
point the reader to \cite{MR2397738}*{\S 6.2}. If we choose $(\Sigma, \omega)$ in the following construction to be $D^{2n}$ with its standard symplectic form 
$d\boldsymbol x\wedge d\boldsymbol y$, and radial Liouville vector field $\frac{1}{2}(\boldsymbol x\thinspace\partial_{\boldsymbol x} + \boldsymbol y\thinspace\partial_{\boldsymbol y})$, the $1$-handle construction described below yields an ordinary 
symplectic $1$-handle as described by Eliashberg \cite{MR1044658} and Weinstein \cite{MR1114405}. 

Consider $\homeOfHandle$ with coordinates $(p,z,t)$ and symplectic form 
$\Omega = \omega + dz\wedge dt$, where $\hat\Sigma$ denotes the completion 
of $(\Sigma, \beta)$ defined at the beginning of \S\ref{sec:concatenation}. The vector field
\[
			Z = Y + 2z\thinspace\partial_z - t\thinspace\partial_t  
\] 
defines a Liouville vector field for $\Omega$.
Notice that $Z$ is gradient like for the function on $\homeOfHandle$ defined by
\[
				g(p,z,t):= \heightOfSigma^2 + z^2 - \frac{1}{2}t^2, 
\] 
where $\heightOfSigma \co \hat\Sigma \to [0,\infty)$ is a $C^\infty$-function on $\hat\Sigma$ 
as specified at the beginning of \S\ref{sec:concatenation}. 
In particular the Liouville vector field $Z$ is transverse to the non-degenerate level sets of $g$ and hence induces contact structures on them.
Denote by $N(\Sigma),N_0(\Sigma) \subset \hat\Sigma \times \R$ the subsets defined by 
\[
			N(\Sigma) := N_\delta(\Sigma) := \{ \heightOfSigma < \delta, z < 1 \} \ \ \mbox{ and } \ \	N_0(\Sigma) := \{ \heightOfSigma = z = 0 \}
\]
endowed with the contact structure induced by $\iota_Z\omega$.
Let $\mathcal{N}=\mathcal{N}_\delta$ and $\mathcal{N}_0$ 
denote the set of points $(p,z, t) \subset g^{-1}(-1)$ which lie on a 
flow line of $Z$ through $ N(\Sigma)  \times\{ t = \pm 1 \} $ and $ N_0(\Sigma)  \times \{  t = \pm 1 \}$ respectively -- both viewed as subsets in $\homeOfHandle$. The set $\mathcal{N}$ is going to play the role of the lower boundary.
We now define our \textbf{generalized symplectic $1$-handle} $H_\Sigma$ as the locus of points $(p,z, t) \in \homeOfHandle$ satisfying the inequality
\[
						-1 \leq g(p, z, t) \leq 1
\]
and lying on a flow line of $Z$ through a point of $\mathcal{N}$.
Since the Liouville vector field $Y$ is transverse to the level sets of $g$, the $1$-form 
\[
				\alpha = \iota_Z\Omega = \beta + 2z \thinspace dt + t \thinspace dz
\]
induces a contact structure on the lower and upper boundary of $H_\Sigma$.
\parskip 0pt

It is possible to perturb $H_\Sigma$ without changing the contact structure on the lower 
and upper boundary as follows. Let $\nu\co \homeOfHandle \to \R$ be an 
arbitrary function and let $H^\nu_\Sigma$ denote the image of $H_\Sigma$ under the time-$1$ 
map of the flow corresponding to the vector field $\nu Z$. Sometimes it is more convenient 
to work with such a perturbed handle.
%
%
\subsection{Attachment and of the handle and its result}
\label{sec:1 handle result}
%
Let $(M,\xi = \ker \alpha)$ be a $(2n + 1)$-dimensional contact manifold.
Suppose we are given a strict contact embedding of $\mathcal{N}$, endowed with the contact structure 
induced by $i_Z\Omega$, into $(M,\xi = \ker \alpha)$.
In the following we will describe the symplectic cobordism $W_{(M,\Sigma)}$ associated to the attachment of the handle $H_\Sigma$.
\parskip 0pt

Note that for each point $x \in \mathcal{N}\setminus \mathcal{N}_0$ there is a point $\mu(x) > 0$ in time such that the time-$\mu(x)$ map of the 
flow corresponding to the Liouville vector field $Z$ maps $x$ to the upper boundary of $H_\Sigma$.
This defines a function $\mu\co \mathcal{N}\setminus \mathcal{N}_0 \to \R^+$ which we may, 
with respect to the above embedding of $\mathcal{N}$ in $(M,\xi = \ker \alpha )$, extend to a non-vanishing function over $M \setminus \mathcal{N}_0$. We 
continue to denote this map $M \setminus \mathcal{N}_0 \to \R^+$ by $\mu$.
Consider the symplectization $\big(\R \times M, d(e^r\alpha)\big)$ and let $[0,\mu] \times (M\setminus \mathcal{N}_0)$ denote the subset defined by
\[
				[0,\mu] \times (M\setminus \mathcal{N}_0) := \big\{ (r,x)  \co 0 \leq  r \leq \mu(x) \big\}.
\]
For any point $(0,x) \in \{0\} \times (M\setminus \mathcal{N}_0)$ 
the time-$\mu(x)$ map of the flow corresponding to the Liouville vector field $\partial_r$ on $(\R \times M, d(e^r\alpha))$ maps
$(0,x)$ to $(\mu (x),x)$.
We define $W_{(M,\Sigma)}$ as the quotient space
\[
		W_{(M,\Sigma)} := \Big( \big(	[0,\mu] \times (M\setminus \mathcal{N}_0) \big)  \sqcup H_\Sigma \Big)/_\sim ,
\]
where we identify  $ (r,x) \in [0,\mu] \times (\mathcal{N}\setminus \mathcal{N}_0)$ with
the image $\psi^Z_r(x) \in H_\Sigma$ of $x \in \mathcal{N}\setminus \mathcal{N}_0$ -- understood as sitting in $H_\Sigma$ -- under the time-$r$ map of the flow corresponding to the Liouville vector field $Z$.
This identification does respect the symplectic forms (cf.~\cite{MR2397738}*{Lemma 5.2.4}) and we indeed end up with an 
exact symplectic cobordism $W_{(M,\Sigma)}$. The concave boundary component $\partial_- W_{(M,\Sigma)}$ is equal to $(M,\xi)$ 
whereas the convex component $\partial_+ W_{(M,\Sigma)}$ equals 
\[
			\#_{H_{\Sigma}}(M,\xi) := (M,\xi) \setminus \big( S^0 \times  N(\Sigma) \big) \cup_\partial \big(D^1 \times \partial \overline{N(\Sigma)}, \eta  \big),
\]
with the obvious identifications, and where $\eta$ denotes 
the kernel of the contact form $i_Y\omega + dt$. We will refer to 
$\#_{H_{\Sigma}}(M,\xi)$ as \textbf{generalized connected sum}.

Suppose that $(\Sigma, \omega, Y)$ is Weinstein and let
$\varphi\colon\thinspace \Sigma \to \mathbb{R}$ denote the associated
exhausting Morse function for which $Y$ is gradient-like. Then $\varphi + z^2 - \frac{1}{2}t^2$ is an
exhausting Morse function for which $Z$ is gradient-like showing that 
$W_{(M,\Sigma)}$ is Weinstein as well.

Let us recap the above discussion on the level of contact manifolds and finish with the following
statement.
\begin{prop}
\label{thm:extended sum}
	There is an exact symplectic cobordism from $(M,\xi)$ to $\#_{H_{\Sigma}}(M,\xi)$.
	Furthermore if $(\Sigma,\omega)$ is Weinstein, then so is the cobordism.
	In particular we have the following. If $(M,\xi)$ admits a symplectic 
	filling, then so does $\#_{H_{\Sigma}}(M,\xi)$. \qed
\end{prop}
%
%
%
 
%
%
\end{document}